\documentclass[12pt]{article}

\usepackage[margin=1.3in]{geometry}
\usepackage[
  bookmarks=true,
  bookmarksnumbered=true,
  bookmarksopen=true,
  pdfborder={0 0 0},
  breaklinks=true,
  colorlinks=true,
  linkcolor=black,
  citecolor=black,
  filecolor=black,
  urlcolor=black,
]{hyperref}
\usepackage{isomath}
\usepackage{amsmath,amsfonts,amssymb,amsthm,enumitem}
\usepackage{mathtools}
\usepackage[capitalize]{cleveref}
\usepackage{nicefrac}

\usepackage{tikz}
\usetikzlibrary{shapes,backgrounds}

\theoremstyle{plain}
\newtheorem{theorem}{Theorem}[section]

\newtheorem{lemma}[theorem]{Lemma}
\newtheorem{claim}[theorem]{Claim}
\newtheorem{example}[theorem]{Example}
\newtheorem{corollary}[theorem]{Corollary}
\newtheorem*{question*}{Question}
\theoremstyle{definition}

\providecommand{\Z}{\mathbb{Z}}

\newcommand{\N}{\mathbb{N}}
\newcommand{\R}{\mathbb{R}}
\providecommand{\eps}{\epsilon}

\newif\ifdraft
\drafttrue

\ifdraft
\newcommand{\shay}[1]{{\color{red}{Shay:~#1}}}

\newcommand{\new}[1]{{{#1}}}
\else
\newcommand{\shay}[1]{}

\fi

\title{On weak $\eps$-nets and the Radon number}
\author{Shay Moran\thanks{Department of Computer Science, Princeton University, Princeton, USA
\texttt{shaym@cs.princeton.edu}. Part of this  research was done while the author was at the Institute for Advanced Study where he was supported by NSF grant CCF-1412958.} \and
Amir Yehudayoff\thanks{Department of Mathematics, Techion-IIT.
\texttt{amir.yehudayoff@gmail.com}. Research supported by ISF grant 1162/15.}}

\begin{document}

\date{}

\maketitle

\begin{abstract}
We show that the Radon number characterizes the existence of weak nets
in separable convexity spaces
(an abstraction of the Euclidean notion of convexity).
The construction of weak nets when the Radon number is finite
is based on Helly's property and on metric properties of VC classes.
The lower bound on the size of weak nets when the Radon number is large relies on 
the chromatic number of the Kneser graph.
As an application, we prove an amplification result for weak $\eps$-nets.
\end{abstract}

\section{Introduction}

Weak and strong $\eps$-nets were defined by Haussler and Welzl as a tool
for fast processing of geometric range queries~\cite{haussler1986epsilon}.
They have been consequently studied in many areas, including
computational geometry, combinatorics, and machine learning,
and they were used in many algorithmic applications,
including range searching and geometric optimization.

An $\eps$-net is a set that pierces all large sets in a given family of sets.
Formally, let $\mu$ be a probability distribution over a domain $X$
and let\footnote{Here and below
we assume that all sets considered are measurable.} $C \subseteq 2^X$
be a family of sets.
A subset $S$ of $X$ is called a {\em weak $\eps$-net} for $C$ over $\mu$
if $S \cap c \neq \emptyset$ for every $c\in C$ with $\mu(c) \geq \eps$.
A subset $S$ is called a {\em strong $\eps$-net} if in addition
$S$ is contained in the support of $\mu$.
We say that $C$ has weak/strong $\eps$-nets 
of size $\beta=\beta(C,\eps)$
if for every distribution $\mu$
there is a weak/strong $\eps$-net for $C$ over $\mu$
of size at most $\beta$
(we stress that $\beta$ may depend on $\eps$, but not on $\mu$).

To illustrate the difference between weak and strong nets,
consider the uniform distribution on $n$ points
on the unit circle in the plane $X=\R^2$,
and the family $C$~to be all convex hulls of subsets of these $n$ points.
Any strong $\eps$-net must contain at least $(1-\eps)n$ points,
but there are weak $\eps$-nets of size $O(\alpha(\eps)/\eps)$,
where $\alpha(\cdot)$ is the inverse Ackermann function~\cite{Alon08weak};
using points inside the unit disc allows to use significantly less points.
In general, weak nets may be much smaller than strong nets.

The main question we address is under what conditions
do weak nets exist.
This question for strong $\eps$-nets is fairly well understood;
the fundamental theorem of statistical learning, which shows that
the VC dimension characterizes PAC learnability, 
also shows that the VC dimension characterizes the existence of strong nets
(see~\cite{haussler1986epsilon,shalev2014understanding} and references within).
We show that the Radon number
characterizes the existence of weak nets in a pretty general setting (viz.~separable
convexity spaces). 

\subsubsection*{Weak nets}

Weak $\eps$-nets were mostly studied in the context of discrete and convex geometry,
where they are related to several deep phenomena.
For example, Alon and Kleitman~\cite{alon1992piercing} used weak nets in their famous solution of the
$(p,q)$-conjecture by Hadwiger and Debrunner~\cite{hadwiger1957variante}.

\new{Barany, Furedi, and Lovasz~\cite{barany1990number} showed that
convex sets in the plane admit weak nets, 
and Alon, Barany, Furedi, and Kleitman~\cite{alon1992point} established a bound of $O(1/\eps^2)$ on their size.}
Recently Rubin improved this bound to $O(1/\eps^{3/2 + \gamma})$,
where~$\gamma>0$ is arbitrarily small~\cite{Rubin18improved}.
Alon, Barany, Furedi, and Kleitman~\cite{alon1992point} 
extended the existence of weak nets for convex sets to all dimensions $d$;
there are weak $\eps$-nets of size at most roughly $(1/\eps)^{d+1}$.
Their proof relies on several results from convex geometry, like
Tverberg's theorem and the colorful Caratheodory theorem.
Chazelle, Edelsbrunner, Grigni, Guibas, Sharir, and Welzl~\cite{chazelle1993improved} later improved the bound 
to at most roughly $(1/\eps)^d$.
Overall, there are at least three different constructions
of weak nets for convex sets.
Matousek~\cite{matousek02lower} showed that any weak $\eps$-net for convex sets in $\R^d$ 
must \new{contain at least $\Omega(\exp(\sqrt{d/2}))$ points} for $\eps\leq1/50$. 
Later, Alon~\cite{alon12lower} proved the first lower bound that is superlinear in $1/\eps$;
this was later improved to $\Omega\bigl((1/\eps)\log^d(1/\eps)\bigr)$ by Bukh, Matousek, and Nivasch~\cite{bukh2011lower}.
Ezra~\cite{Ezra10} constructed weak $\eps$-nets for the more restricted class of axis-parallel boxes in $\R^d$.
Alon, Kalai, Matousek, and Meshulam~\cite{alon2002transversal}
defined weak nets in an abstract setting,
and asked about combinatorial conditions that yield their existence.

We identify that the existence of weak nets
follows from a basic combinatorial property of convex sets,
Radon's theorem:
 
\begin{theorem}[Radon]
Any set of $d+2$ points in $\R^d$  can be partitioned into
two disjoint subsets whose convex hulls intersect.
\end{theorem}

We show that this property alone is sufficient and necessary 
for the existence of weak nets in a general setting, 
which we describe next.

It is worth mentioning that Radon's theorem 
also plays a central role in the context of strong nets.
Indeed, it implies that the VC dimension of
half-spaces in $\R^d$ is at most $d+1$,
which consequently bounds the VC dimension
of many geometrically defined classes.

\subsubsection*{Convexity spaces}

We consider \new{an abstraction} of Euclidean convexity that originated
in a paper by Levi~\cite{levi1951helly}, and defined in the form \new{presented} here
by Kay and Womble~\cite{kay1971axiomatic}. 
For a thorough introduction to this subject see
the survey by Danzer, Grunbaum, and Klee~\cite{danzer1963helly} 
or the more recent book by van de Vel~\cite{van1993theory}.

A {\em convexity space} is a pair $(X,C)$ where $C \subseteq 2^X$ is a
family of subsets that satisfies\footnote{\new{Note that van de Vel~\cite{van1993theory} also requires that the union of an ascending chain of convex sets is convex.}}:
\begin{itemize}
\item $\emptyset,X\in C$.
\item $C$ is closed under intersections\footnote{We use the standard notation
$\cap C' = \bigcap_{c \in C'} c$.}: $\cap C' \in C$ for every $C'\subseteq C$.
\end{itemize}
The members of $C$ are called \emph{convex sets}.
The \emph{convex-hull} of a set $Y \subseteq X$, denoted by $conv(Y) = conv_C(Y)$, is
the intersection of all convex sets $c \in C$ that contain $Y$.
A convex set $b\in C$ is called a \emph{half-space} if its complement is also convex.


We next define the notion of separability, which is an abstraction of the
hyperplane separation theorem (and the more general Hahn-Banach theorem).
The convexity space $(X,C)$ is \emph{separable}, if for every $c\in C$ and $x \in X \setminus c$
there exists a half-space $b \in C$ so that $c \subseteq b$ and $x \not \in b$.
It can be verified that $(X,C)$ is separable if and only if 
every convex set $c\in C$ is the intersection of all
half-spaces containing it. 
This form of separability, as well as other forms,
have been extensively studied (e.g.~\cite{GrunbaumMotzkin61,hammer1955,hammer1961semispaces,Klee56,Chepoi1994}).

Convexity spaces appear in many contexts in mathematics. 
For instance, the family of closed subsets in a topological space, 
the subgroups of a given groups, and the subrings of a ring are all examples\footnote{One should sometimes add the empty set in order to satisfy all axioms of a convexity space.} of convexity spaces.
They are also closely related to the notion of $\pi$-systems in probability theory.
In Section~\ref{sec:Example} below we discuss a few examples of convexity spaces that arise in algebra and combinatorics.


\subsubsection*{Main results}

The combinatorial property that characterizes the existence
of weak nets is the Radon number~\cite{levi1951helly,kay1971axiomatic},
{which is an abstraction of Radon's theorem:
we say that $C$ \emph{Radon-shatters} a set $Y\subseteq X$ if for every partition of $Y$ 
into two parts $Y_1,Y_2$ it holds that $conv(Y_1)\cap conv(Y_2)=\emptyset$.
The {\em Radon number} of $(X,C)$ is
the minimum number~$r$ such that $C$ does not Radon-shatter \new{any} set of size~$r$.
Radon's theorem states that the Radon number of 
the space of convex sets in $\R^d$ is at most~$d+2$.

%

\begin{theorem}\label{thm:main}
Let $(X,C)$ be a finite separable convexity space.
\begin{enumerate}
\item If the Radon number of $(X,C)$ is at most $r$ then
it has weak $\eps$-nets of size at most $(120 r^2/\eps)^{4r^2\ln(1/\eps)}$ for every $\eps > 0$.
\item 
If the Radon number of $(X,C)$ is more than $r$
then there is a distribution $\mu$ over $X$ such 
that every $\frac{1}{4}$-net for $C$ over $\mu$ has size at least $r/2$.
\end{enumerate}
\end{theorem}

We often refer to a construction of small weak $\eps$-nets
as an {\em upper bound}, and to a proof that no small weak $\eps$-nets
exists as a {\em lower bound}.
The upper bound in 1 above is quantitively worse than the one for Euclidean convex sets~\cite{chazelle1993improved},
but holds in a more general setting.
We do not know what is the optimal bound in this generality.

A possible interpretation of the upper bound is that the existence
of weak nets is not directly related to ``geometric'' properties of the underlying space
(as in the various constructions surveyed above).
This is somewhat surprising:
consider a family $C$ and a distribution $\mu$
such that $C$ has no strong $\eps$-net
with respect to $\mu$; in order to construct a weak $\eps$-net
we should have a mechanism that suggests ``good points''
outside the support of $\mu$. In Euclidean geometry 
there are such natural choices, like ``center of mass''.
We notice that the Radon number provides such a mechanism
(see Section~\ref{sec:overview}).



We next discuss conditions
that are equivalent to the existence of small $\eps$-nets.
It is convenient to present these equivalences for infinite spaces.
Quantitative variants of these statements apply to finite spaces as well.

We use the following standard notion of compactness; a family $C$ is  
\emph{compact} if for every $C' \subseteq C$
so that $\cap C' = \emptyset$ there is a finite $C'' \subseteq C'$
so that $\cap C'' = \emptyset$.
This condition is satisfied e.g.\ by closed sets in a compact topological space.

\begin{corollary}[Equivalences]
\label{cor:EQ}
The following are equivalent for a compact separable convexity space $(X,C)$:
\begin{enumerate}
\item $(X,C)$ has a finite Radon number.
\item $(X,C)$ has weak $\eps$-nets of finite size for some $0 < \eps <1/2$.
\item $(X,C)$ has weak $\eps$-nets of finite size for every $\eps >0$.
\end{enumerate}
\end{corollary}

The proof of Corollary~\ref{cor:EQ} appears in Section~\ref{sec:corEQ}.
It shows that the role of the Radon number in the existence of weak nets 
is similar to the role of the VC dimension in the existence of strong nets,
at least for separable compact convexity spaces.

In Section~\ref{sec:NecAssum}, we provide an example 
showing that the compactness assumption in Corollary~\ref{cor:EQ}
is necessary. We do not know if the separability assumption is necessary.

The implication $2 \Rightarrow 3$ in Corollary~\ref{cor:EQ} is an amplification
statement for the parameter~$\eps$  in weak nets.
In Section~\ref{sec:boost} we give an example showing that
the threshold $1/2$ in item 2 is sharp for amplification of weak nets:
\begin{example}
\label{ex:boost}
%
There is a compact separable {convexity} space that has
weak $\eps$-nets of finite size for every $\eps > 1/2$ 
but has no weak $\eps$-nets of finite size for $\eps<\frac{1}{2}$.
\end{example}

This demonstrates an interesting difference with
strong $\eps$-nets, for which there is no such threshold:
if $C$ has strong $\eps$-nets for some $\eps<1$
then it has finite VC dimension, which implies
{the existence of} strong $\eps$-net for all $\eps >0$.


%

%

\subsubsection*{Organization}

In Section~\ref{sec:Example} we provide some examples
of convexity spaces.
In Section~\ref{sec:overview} we outline 
the construction that leads to the upper bound in Theorem~\ref{thm:main},
in Section~\ref{sec:LBintro} we outline the lower bound in Theorem~\ref{thm:main},
and in Section~\ref{sec:NecAssum} we provide some examples
that demonstrate the necessity of some of our assumptions.

In Section~\ref{sec:construction} we prove the upper bound,
and in Section~\ref{sec:RadonLB}  we prove the lower bound.
In Section~\ref{sec:cors} we prove the {characterizations (Corollary~\ref{cor:EQ}),} 
and 
in Section~\ref{sec:ext} we discuss an extension
to convexity spaces that are not necessarily separable or compact (like {bounded} convex sets in $\R^d$).
Finally, in Section~\ref{sec:discussion} we conclude the paper and offer some directions for future research.

\subsection{Some convexity spaces}
\label{sec:Example}
We now present a few examples of ``non Euclidean'' convexity spaces.
These examples will be used later on to show that some of our theorems/lemmas
are tight in the sense that each premise is necessary.
More examples can be found in the book~\cite{van1993theory}.

\vspace{0.3cm}
\noindent {\bf Example 1 (power set).}
Let $X$ be a set. Perhaps the simplest convexity space is~$(X,2^{X})$.
Here every convex set is a half-space and therefore this space is separable.
When $X$ is finite, the Radon number of this space is $\lvert X\rvert + 1$.
When $X$ is infinite, the Radon number is $\infty$.

\vspace{0.3cm}
\noindent {\bf Example 2 (subgroups).}
Let $G$ be a group with identity $e$.
The space $(G\setminus\{e\}, \{H\setminus\{e\} : H\leq G \})$
of all subgroups of $G$ (with the identity removed)
is a convexity space. 
Here, $Y\subseteq G$ is Radon-shattered,
if every two disjoint subsets of~$Y$ generate groups
whose intersection is $\{e\}$.

\vspace{0.3cm}
\noindent {\bf Example 3 (cylinders).}
Let $X = \{0,1\}^n$,
and let the $C$ be the family of \emph{cylinders}:
a set $c\subseteq \{0,1\}^n$ is called a cylinder if there exists $Y\subseteq [n]$,
and $v\in\{0,1\}^Y$ such that 
\[c = \bigl\{u\in\{0,1\}^n : u|_Y = v\bigr\}.\]
The size of $Y$ is called the co-dimension of the cylinder.
This is a separable convexity space.
Its half-spaces are the cylinders with co-dimension $1$,
and its Radon number is $\Theta(\log n)$.

\vspace{0.3cm}
\noindent {\bf Example 4 (subtrees).}
Let $T=(V,E)$ be a finite tree. Consider the convexity space $(V,C)$, where
\[C = \bigl\{U\subseteq V : \text{the induced subgraph on $U$ is connected}\bigr\}.\]
It is a separable convexity space and its Radon number is at most
$4$. 
Theorem~\ref{thm:main} hence implies the existence of weak nets
of size depending only on $\eps$
(in this case there are elementary constructions of $\eps$-nets of size~$O(1/\eps)$).
%
%
%
This example is a special case of \emph{geodesic convexity}
in metric spaces (see~\cite{van1993theory}).


\vspace{0.3cm}
\noindent {\bf Example 5 (convex lattice sets).}
Consider the space $(\Z^d, C)$, where $C$ is the family of \emph{convex lattice sets} in $\R^d$;
these are sets of the form {$K \cap \Z^d$ for some convex~$K \subseteq \R^d$.}

This is a separable
convexity space. {Indeed, 
let $c = K \cap \Z^d \in C$ and $x \in \Z^d \setminus c$.
Since $x \not \in c$ it follows that $x \not \in K$.
and therefore there is a half-space in $\R^d$ separating~$x$ and $K$. 
This half-space induces a half-space in $(\Z^d, C)$.}

Onn~\cite{Onn91} proved that the Radon number of
this space is at most $O(d2^{d})$ {and at least $2^d$.} 
Our results imply that weak $\eps$-nets 
exist in this case as well.

Note that the family of half-spaces has VC dimension $d+1$,
which is much smaller than the Radon number. 
Thus, Theorem~\ref{thm:weakeps} (which we state in the next subsection) 
gives better bounds on the size of the $\eps$-net than Theorem~\ref{thm:main}. 

\vspace{0.3cm}
\noindent {\bf Example 6 (linear extensions of posets).}
Let $\Omega$ be a set. 
For a partial order $P$ on $\Omega$, 
let $c(P)\subseteq X$ denote the set of all linear orders that
extend $P$.
Fix a partial order $P_0$, and
consider the family $C$ of 
all sets of the form $c(P)$, where
$P$ is a partial order that extends $P_0$.
The space $\bigl(c(P_0),C\bigr)$ is a separable convexity space whose half-spaces
correspond to partial orders defined
by taking two $P_0$-incomparable elements 
$x,y\in \Omega$ and extending $P_0$ by setting $x<y$.

\subsection{Upper bound}
\label{sec:overview}

\subsubsection*{{From Radon to Helly and VC}}
Let $(X,C)$ be a separable convexity space and let $B$ denote the family of half-spaces of $C$
(the notation $B$ is chosen to reflect that $B$ generates all convex sets in $C$ by taking intersection, and hence
can be seen as a \emph{basis}).

We first observe that the Radon number is an upper bound on the {Helly number}
and the {VC dimension} of $B$.
The {\em Helly number} of a family $B$ is the minimum number $h$ such that
every finite\footnote{The finiteness assumption can be removed when $B$ is compact.} $B'\subseteq B$ 
with $\cap B'=\emptyset$ \new{contains a subfamily $B''$ with
at most $h$ sets such that already the intersection of the sets in  $B''$ is empty}.
Helly theorem states that the Helly number of half-spaces in $\R^d$ is at most $d+1$.
The {\em VC dimension} of $B \subseteq 2^X$ is the supremum over~$v$ for which
there exists $Y \subseteq X$ of size $|Y|=v$ so that for every $Z\subseteq Y$
there is a member of the family that contains $Z$ and is disjoint
from $Y\setminus Z$.


\begin{lemma}
\label{lem:Radon}
Let $B$ the family of half-spaces of a separable convexity space $C$.
If the Radon number of $C$ is $r$
then the VC dimension and the Helly number of $B$ are less than $r$.
\end{lemma}

The bound on the Helly number follows from a result by Levi~\cite{levi1951helly},
and the bound on the VC dimension is straightforward.
The proof appears in Appendix~\ref{ap:Radon}.

%

Lemma~\ref{lem:Radon} reduces the construction of weak nets for separable convexity spaces
to the following, more general construction.

\begin{theorem}\label{thm:weakeps}
Let $X$ be a set, let $B\subseteq 2^X$ be a compact family
with $VC$ dimension~$v$ and Helly number $h$,
and let $B^\cap$ denote the family generated by taking arbitrary
intersections of members of $B$.
Then $B^{\cap}$ has weak $\eps$-nets of size at most $\beta<(120 h^2/\eps)^{4hv\ln(1/\eps)}$ for every $\eps>0$.
\end{theorem}
%

We next give an overview of
the proof of Theorem~\ref{thm:weakeps};
the complete proof appears in Section~\ref{sec:construction}.


%


\subsubsection*{Outline of construction}

The construction of weak nets underlying Theorem~\ref{thm:weakeps} is short and simple.
We want to pierce all convex sets $c \in C$ such that $\mu(c) \geq \eps$.
We distinguish between two cases.
The simpler case is when $c$ can be written as
the intersection of half-spaces, each of which has $\mu$-measure more than $1-1/h$.
By Helly's property, it follows that there is a single point
that pierces all such $c$'s.
In the complementary case, when~$c$ can not be written in this way,
we use Haussler's packing lemma~\cite{haussler1995sphere} and show that there is
a small collection $A \subseteq B$ such that conditioning $\mu$ on a single~$a \in A$ 
increases the measure of $c$ by a factor of at least $1+1/(2h)$.
The size of $A$ can be bounded from above in terms of the VC dimension of $B$.
So, constructing a set that pierces all $c$'s with measure at least $\eps$
is reduced to constructing a bounded number of nets for larger density $\eps' \geq (1+1/(2h))\eps$.

To summarize, the Helly number yields the theorem for $\eps$ close to $1$,
and the VC dimension allows to keep increasing the density until
the Helly number becomes relevant.

Going back to the discussion after Theorem~\ref{thm:main} concerning the mechanism that suggests ``good points'' we see that this mechanism is based on the Helly property
(roughly speaking, the Helly property is a mechanism that given a collection of sets
outputs a point).

\subsection{Lower bound}
\label{sec:LBintro}

The following lemma is a slight generalization of the lower bound in Theorem~\ref{thm:main}
(we do not assume separability or compactness, and replace $1/4$ by any $\eps>0$).

\begin{lemma}
\label{lem:LB}
Let $(X,C)$ be a convexity space.
If the Radon number of $C$ is greater than $r>0$ 
then there is a distribution $\mu$ on $X$
so that every weak $\eps$-net for $C$ over~$\mu$
has size at least $(1-2\eps)r$.
\end{lemma}

This gives a non-trivial lower bound as long as $\eps < 1/2$.
Example~\ref{ex:boost} shows that this is sharp in the sense that
when $\eps>1/2$ there is no lower bound that tends to infinity with the Radon number.

The proof of Lemma~\ref{lem:LB}, 
as well as a finer distribution dependent lower bound,  
appear in Section~\ref{sec:RadonLB}.
The proof is essentially by reduction to
the chromatic number of Kneser graphs.

\subsection{The necessity of assumptions}
\label{sec:NecAssum}

\subsubsection*{Assumptions in Theorem~\ref{thm:weakeps}}
We first show that if either of the assumptions of having bounded VC dimension
or of having bounded Helly number is removed then Theorem~\ref{thm:weakeps}
ceases to hold.

To see why bounded Helly number is necessary, 
let $X$ be any finite set and set $B=\{X\}\cup\{X\setminus\{x\} : x\in X\}$.
The VC dimension of $B$ is $1$, 
but any subset of $X$
can be represented as an intersection of members of $B$. 
Thus, $B^{\cap} = 2^X$, 
which does not have weak $\eps$-nets of size which is
independent of $\lvert X\rvert$.

To see why bounded VC dimension is necessary,
consider the convexity space $(X,C)$ of cylinders
(Example 3 in Section~\ref{sec:Example}).
Here, $X=\{0,1\}^n$, $C$ is the family of cylinders,  
and the family of half-spaces $B$ consists of cylinders with co-dimension~1:
$B = B_0 \cup B_1$ with
\[B_t = \bigl\{\{x\in X : x_i=t\} : i\in[n]\bigr\}.\]
The Helly number of $B$ is $2$,
since an intersection of half-spaces
is empty if and only if two complementing
cylinders with co-dimension 1 participate in it.

The following claim gives a lower bound on weak $\frac{1}{4}$-nets
for $C$ over the uniform distribution $\mu$ over $X$.

\begin{claim}\label{claim:cylinders}
Every weak $(1/4)$-net for $C$ over $\mu$ has size at least $\log n$.
\end{claim}
\begin{proof}
$S\subseteq X$ pierces every cylinder with measure $1/4$ 
only if for every $i \neq j$ in $[n]$ there is $x\in S$ with $x_i=0$ and $x_j=1$.
Now, consider the mapping from $[n]$ to $\{0,1\}^S$, which maps $i\in[n]$
to $(x_i)_{x\in S}$. By the above, this mapping is one-to-one, and in particular
$2^{\lvert S\rvert}=\bigl\lvert \{0,1\}^S\bigr\rvert \geq n$.
\end{proof}

\subsubsection*{Assumptions in Corollary~\ref{cor:EQ}}

We now describe an example showing that
the compactness assumption in Corollary~\ref{cor:EQ} is necessary
(a related example appears in~\cite{BEHW89,wenocur1981some} in the context of strong $\eps$-nets).
Let $X=\omega_1$ be the first uncountable ordinal,
and $C$ be the family of all intervals in the well-ordering of $X$
(a set $I\subseteq X$ is an interval if whenever $a,b\in I$ and $a\leq x\leq b$ then also $x\in I$). 
The space $(X,C)$ is separable with Radon number~$3$,
but is not compact.

We claim that it does not have finite weak $\eps$-nets, even for $\eps=1$. 
Indeed, let~$\mu$ be the probability distribution 
defined over the $\sigma$-algebra generated by countable subsets of $X$
and assigns every countable subset of $X$ measure $0$.
Every interval is either countable or has a countable complement, and is therefore measurable. 

We now claim that there is no finite $S\subseteq X$ that pierces
all intervals of measure~$1$. Indeed, let $S$ be finite, and let $m$ be the maximum element in $S$. 
The interval $\{x\in X : x > m\}$ has measure $1$
but is not pierced by $S$.

%
%

\section{Proof of upper bound}
\label{sec:construction}

Here we construct weak $\eps$-nets when the Helly number
and the VC dimension of the half-spaces are bounded (Theorem~\ref{thm:weakeps}).
The property of VC classes that we use is 
the following packing lemma due to Haussler~\cite{haussler1995sphere}.

\begin{theorem}[Haussler]
\label{thm:packing}
Let $B\subseteq 2^X$ be a class of VC dimension $v$.
For every distribution $\mu$ on $X$ and for every $\delta > 0$,
there is $A \subseteq B$ of size $\lvert A\rvert \leq (4e^2/\delta)^v$
such that for every $b \in B$ there is $a\in A$
with $\mu(a\Delta b)\leq \delta$.
\end{theorem}

Haussler's stated the lemma
in a dual way; the number of disjoint balls of a given radius in a VC class
is small.
Haussler's proof is elaborate, but a weaker bound can be proved fairly easily.
Indeed, consider a finite set $A$ so that
$\mu (a \Delta a') > \delta$ for all $a \neq a'$ in $A$ .
Let $x_1,\ldots,x_m$ be $m$ independent samples from $\mu$ for \new{$m \geq 2 \log (|A|)/\delta$}.
Let $Y = \{x_1,\ldots,x_m\}$,
and let $A|_Y = \{a \cap Y : a \in A\}$.
On one hand, the Sauer-Shelah-Perles lemma implies that
$A|_Y$ is small:
$\bigl\lvert A|_Y\bigr\rvert \leq (e m / v)^v$.
On the other hand, by the union bound, 
$\bigl\lvert A|_Y\bigr\rvert=|A|$ with positive probability.
This implies that $A$ is small.

\begin{proof}[Proof of Theorem~\ref{thm:weakeps}]
We start by focusing on the set
$B_0 = \{b\in B : \mu(b) > 1-1/h\}$.
By the union bound, every $h$ members of $B_0$ intersect.
Since $B_0\subseteq B$ is compact with Helly number $h$, 
there is a single point $x_0 \in X$ that pierces all sets in~$B_0$.

Let $0 < \eps <1$.
We construct the  $\eps$-net by induction on $N(\eps)$,
which is defined to be the minimum integer $n$ such that
$\eps\bigl(1+1/(2h)\bigr)^n > 1-1/h$.


\subsubsection*{Induction base}
If $N(\eps)=0$, define the piercing set $S = S(\mu,\eps)$ as
$$S = \{x_0\},$$
where $x_0$ is the point that pierces all half-spaces in $B_0$.
Indeed, $S$ is an $\eps$-net as every $c\in C$ with $\mu(c)\geq \eps > 1-1/h$
{is the intersection of half-spaces} from $B_0$, so $x_0$ pierces $c$ as well.

\subsubsection*{Induction step}
Let $1 >\eps > 0$ such that $N(\eps) > 0$.
We construct the piercing set $S=S(\mu,\eps)$ as follows:
Set $\delta = \eps/(2h)^2$ and pick some $A\subseteq B$ as in Theorem~\ref{thm:packing}. 
Also set $\eps'=(1+1/(2h))\eps$.
Note that $N(\eps')= N(\eps)-1$.
By induction,
for each $a \in A$ with $\mu(a) > 0$, 
pick a piercing set $S_a = S(\mu|_{a},\eps')$, where 
$\mu|_{a}$ denotes the distribution $\mu$ conditioned on $a$.
Finally, let 
$$S=\{x_0\}\cup\bigcup_{a \in A : \mu(a)>0}S_a.$$
It remains to prove that $S$ satisfies the required properties.

\

\noindent
\underline{$S$ is piercing:}
Let $c \in C$ with $\mu(c) \geq \eps$.
If $c$ is generated\footnote{\new{I.e.\ $c$ can be presented as an intersection of sets from $B_0$.}} by $B_0$ then $x_0$ pierces $c$.
Otherwise, there is some $b\in B$ with $\mu(b) \leq 1-1/h$ that contains $c$.
Pick $a\in A$ such that $\mu(a\Delta b)\leq \delta$. 
Since $\mu(b) \geq \eps$ and $\mu(a\Delta b)\leq\delta$,
$$\mu(a) \geq \mu(b) - \mu(b\setminus a) \geq 
\eps- \delta > 0,$$
which means that $\mu|_a$ is well-defined.
We claim that
$S(\mu|_{a},\eps')$ pierces $c$. To this end, it suffices to show that $\mu|_{a}(c)\geq \eps'$:
\begin{align*}
\mu|_{a}(c) &= \frac{\mu(c\cap a)}{\mu(a)}\\
                   &= \frac{\mu(c)- \mu\bigl(c\cap(b\setminus a)\bigr)}{\mu(a)}\tag{$c\subseteq b$}\\
                   &\geq \frac{\mu(c)- \mu(b\setminus a)}{\mu(a)}\\
                   &\geq \frac{\eps- \delta}{1-1/h + \delta}
                   \tag{$\mu(c) \geq \eps, \mu(b) \leq 1-1/h,\mu(a\Delta b)\leq\delta$}\\
                 &\geq \frac{\eps(1-1/(2h)^2)}{1-1/(2h)} \tag{$\delta=\eps/(2h)^2$}\\
                   &=\eps'.
\end{align*}

\

\noindent
\underline{\em $S$ is small:}
Let $\beta(n)$ denote the maximum possible size of $S$ for $\eps$ with $N(\eps)=n$.
The argument above yields
\[ \beta(n) \leq 1 + (4e^2/\delta)^v \cdot\beta(n-1) =   1 + (16 e^2  h^2/\eps)^v\cdot\beta(n-1).\]
Since $\beta(0) = 1$, we get

$$\beta(n) \leq (120 h^2 / \eps)^{vn}.$$
Finally, since $N(\eps) \leq 4 h\ln(1/\eps))$, we get the bound
$$|S(\mu,\eps)| \leq (120 h^2/\eps)^{4hv\ln(1/\eps)}.$$


\end{proof}

\section{Proof of lower bound}
\label{sec:RadonLB}

The proof of Lemma~\ref{lem:LB} is based on the following distribution-dependent
lower bound on the size of weak nets:

\begin{lemma}\label{lem:LBDD}
Let $C$ be a family of subsets over a domain $X$ and let $\mu$ be a distribution on $X$.
For $\eps>0$ define a graph $G=G(\mu,\eps)$ whose vertices are the sets $c\in C$
such that $\mu(c)\geq \eps$, and two sets are connected by an edge if and only if they
are disjoint.
Then, every weak $\eps$-net for $C$ over $\mu$ has size at least the chromatic
number of $G$, which is denoted by $\chi(G)$.
\end{lemma}
\begin{proof}
Let $S$ be a set that pierces all $c\in C$ with $\mu(c)\geq\eps$.
Define a coloring of $G$ by assigning to every vertex $c$ an element $x\in S\cap c$.
This is a proper coloring of $G$, 
since if $\{c,c'\}$ is an edge of $G$
then $c$ and $c'$ are disjoint and therefore can not be pierced by the same element of $S$.
\end{proof}

Lemma~\ref{lem:LBDD} is tight whenever the class $C$ has Helly number $2$
(like in examples~3 and~4 in Section~\ref{sec:Example}).
Indeed, consider an optimal coloring of $G(\mu,\eps)$.
Every color class is an independent set in $G$ 
(which means that every two sets in it have a non-empty intersection). So,
when the Helly number is $2$, each color class
can be pierced by a single element, and we get
a piercing set of size $\chi(G)$.

The proof Lemma~\ref{lem:LB} thus reduces to a lower bound on the chromatic
number of the relevant graph, which in our case contains a copy of the Kneser graph.
The {\em Kneser graph} $KG_{n,k}$ is the graph whose vertices correspond to the k-element subsets of a set of n elements, and where two vertices are adjacent if and only if the two corresponding sets are disjoint.

Lovasz~\cite{lovasz1978kneser} proved Kneser's conjecture on the chromatic
number of this graph (this proof is considered seminal 
in the topological method in combinatorics):

\begin{theorem}[Lovasz]
\label{thm:Lovasz}
The chromatic number of $KG_{n,k}$ is $n-2k+2$.
\end{theorem}

We actually do not need the full strength of Lovasz's result.
A lower bound of the form $\chi(KG_{n,n/4}) \geq n/10$
suffices for deducing the equivalence between the existence
of weak nets and finite Radon number
(in fact, even much weaker bounds suffice).
Noga Alon informed us that for this range of the parameters
there is a short and elementary proof~\cite{AlonPrivate}.
Since this argument does not appear in the literature 
and may be useful elsewhere, we include Alon's proof in Section~\ref{sec:Alon}.

\begin{proof}[Proof of Lemma~\ref{lem:LB}]
Since the Radon number is greater than $r$,
it follows that there is a set $Y\subseteq X$
of size $r$ that is Radon-shattered by $C$.
Pick $\mu$ to be the uniform distribution over $Y$.
By Lemma~\ref{lem:LBDD} it suffices to show that
$\chi\bigl(G(\mu,\eps)\bigr)\geq (1-{2\eps})r$.
This follows by noticing that since $Y$ is Radon-shattered, it follows that
the subgraph of $G(\mu,\eps)$ induced by the vertices
$conv(Z)$, for $Z\subseteq Y$ is of size $\lceil {\eps r}\rceil$,
is isomorphic to $KG_{r,\lceil {\eps r}\rceil}$.


\end{proof}

\section{Proof of equivalences}
\label{sec:cors}

\subsection{The existence of weak nets}
\label{sec:corEQ}

\begin{proof}[Proof of Corollary~\ref{cor:EQ}]
Let $(X,C)$ be a compact separable convexity space.

\paragraph{$1 \Rightarrow 3.$}
By Lemma~\ref{lem:Radon}, the family $B$ of half-spaces of $C$
has finite VC dimension and Helly number.
Theorem~\ref{thm:weakeps} now implies that $B^{\cap}=C$
has finite weak $\eps$-nets for every $\eps>0$.

\paragraph{$3\Rightarrow 2.$} Obvious.

\paragraph{$2 \Rightarrow 1.$}
Assume that $C$ has weak $\eps$-nets of size $\beta = \beta(\eps)<\infty$ for some $\eps < 1/2$.
By Lemma~\ref{lem:LB}, the Radon number of $(X,C)$
is at most $\frac{\beta}{1-2\eps}$.
\end{proof}

\subsection{The threshold $1/2$ is sharp}
\label{sec:boost}

Here we describe Example~\ref{ex:boost}.
Let $X=\{0,1\}^{\N}$ be the Cantor space, 
and let $C$ be the family of all cylinders;
recall that $c\subseteq X$
is a cylinder if there exist $Y\subseteq \N$
and $u\in\{0,1\}^Y$ such that $c=\{v\in X : v|_Y = u\}$.
The space $(X,C)$ is a separable convexity space.
It is also compact
(this follows e.g.\ from Tychonoff's theorem).

We claim that $(X,C)$ has $\eps$-nets of size $1$ for every $\eps > 1/2$.
This follows since for every distribution $\mu$ the family $\{c\in C: \mu(c) > 1/2\}$ is intersecting\footnote{A family of sets is {\em intersecting} if every two members of it intersect.},
and since $C$ has Helly number $2$.
Hence, $\cap \{c\in C: \mu(c) > 1/2\} \neq \emptyset$ for all $\mu$, as claimed.

It remains to show that $(X,C)$ has no finite weak $\eps$-nets
for $\eps < 1/2$. By Corollary~\ref{cor:EQ},
it suffices to consider the case $\eps=1/4$.
Let $\mu$ be the Bernoulli measure on the Cantor space;
namely the infinite product of uniform measure on $\{0,1\}$.
Now, $S\subseteq X$ is a weak $\eps$-net over $\mu$ if and only if
it intersects every cylinder with co-dimension $2$. In particular
for every $i \neq j$ in $\N$ there must be $x\in S$ with $x_i=0$ and $x_j=1$.
By the proof of Claim~\ref{claim:cylinders} it follows that such an $S$ must be infinite.

\section{An extension}
\label{sec:ext}

Consider the space of bounded closed
convex sets in $\R^d$.
The corresponding convexity space in not separable nor compact.
Nevertheless, our results extend to this space as well.

The following variants of separability and compactness suffice.
A convexity space is called {\em locally-separable} if
there is a set $B \subseteq C$ so that
$C = B^\cap$ and for every finite $Y \subseteq X$ and for every $b \in B$
there is $\bar b \in B$ so that $b \cap Y$ and $\bar b \cap Y$
form a partition of $Y$.
Every separable space is locally-separable,
but there are convexity spaces 
which are locally separable but not separable
(like the space of bounded convex sets in $\R^d$).
A convex set $c\in C$ is called {\em compact}
if the restricted convexity space $(c,\{c'\cap c : c'\in C\})$ is compact.
The {\em Radon number} of $c$ is the Radon number of 
the space $(c,\{c'\cap c : c'\in C\})$.

\begin{theorem}\label{thm:ext}
Let $(X,C)$ be a locally separable convexity space such that
there exists a chain $c_1\subseteq c_2\subseteq\ldots$
of compact convex sets each of which has Radon number at most $r$ 
such that $\bigcup_i c_i = X$.
Then $(X,C)$ has finite weak $\eps$-nets for every $\eps>0$.
\end{theorem}


Theorem~\ref{thm:main} and its proof apply
for locally-separable and compact convexity spaces. 
Theorem~\ref{thm:ext} therefore follows 
by applying it to a $c_i$ in the chain
such that $\mu(c_i)\geq 1-\eps/2$.

\section{Future research}
\label{sec:discussion}

We showed that for compact separable convexity spaces,
the existence of weak nets is equivalent to having a finite
Radon number. We now suggest several directions for future research.

One interesting direction is to find a characterization that is valid
even more generally (i.e.\ for families that are not necessarily separable convexity spaces).
It is worth noting in this context that the definition of 
the Radon number can be extended to arbitrary families.
One may extend the definition of Radon-shattering as follows:
A family $C\subseteq 2^X$ \emph{Radon-shatters} 
the set $Y\subseteq X$ if for every partition of $Y$ 
into two parts $Y_1,Y_2$ there are two disjoint sets 
$c_1,c_2\in C$ such that $c_1 \cap Y = Y_1$ and $c_2 \cap Y = Y_2$.

This extension of the Radon number does not characterize the existence
of weak nets for arbitrary families.
For instance, let $C=\{Y\subseteq [n] : \lvert Y\rvert > n/2\}$.
The Radon number is $2$ since every two sets in $C$ intersect,
but every weak $\frac{1}{2}$-net over the uniform distribution
has size at least $n/2-1$.
However, this family $C$ is not convex (closed under intersection),
which is a crucial property in our work.
It may also be interesting to fully understand the role
of convexity in the context of weak nets.

Alon et al.~\cite{alon2002transversal}
studied weak nets and the $(p,q)$-property in an abstract
setting and described connections to fractional Helly properties.
It may further be interesting to investigate which other combinatorial
properties of convex sets apply in more general settings. 

An additional question that comes to mind
is the dependence of the size of weak $\eps$-nets
on $\eps$.
In this direction, Bukh, Matousek and Nivasch
proved an $\Omega \left( \frac{1}{\eps} \log^{d-1} \frac{1}{\eps} \right)$
lower bound on the size of weak $\eps$-nets for convex sets in $\R^d$~\cite{bukh2011lower},
and recently Rubin~\cite{Rubin18improved} proved an upper bound of roughly $O(1/\eps^{3/2})$ in $\mathbb{R}^2$
(which improves upon the general bound in $\mathbb{R}^d$ of roughly $O(1/\eps^d)$ by~\cite{chazelle1993improved}).
Proving tight bounds on the size of weak $\eps$-nets
is a central open problem in this area.
The general framework developed here may be useful
in proving stronger lower bounds.

%
%
%
%
%
%
%

\subsection*{Acknowledgements}

We thank Noga Alon, Yuval Dagan, and Gil Kalai for helpful conversations.
We also thank the anonymous reviewers assigned by SoCG '19 for their helpful comments which
improved the presentation of this work.
\bibliographystyle{abbrv}
\bibliography{ref}

\appendix

%
%

\section{Radon, Helly and VC}
\label{ap:Radon}
Here we prove that the Radon number bounds from above both the Helly number
and the VC dimension (Lemma~\ref{lem:Radon}).
The proof follows from the following two claims.
Levi~\cite{levi1951helly} proved that

\begin{claim}
Let $C$ be a convexity space.
If the Radon number of $C$ is $r$ then
its Helly number is smaller than $r$.
\end{claim}

\begin{proof}[Proof (for completeness)]
Let $C' \subseteq C$ be a finite family so that $\bigcap_{c \in C'} c = \emptyset$.
Let $K\subseteq C'$ be a minimal subfamily so that $\bigcap_{c\in K} c = \emptyset$.
Assume towards a contradiction that $\lvert K \rvert \geq r$.
Minimality implies that for each $k\in K$ we have 
$C_k : = \bigcap_{c\in K \setminus\{k\}} c \neq \emptyset$.
Let $x_k \in C_k$.
The $x_k$'s must be distinct (otherwise $\bigcap_{c\in K} c \neq \emptyset$).
Thus, there is a partition of $\{x_k : k\in K\}$ to two parts $Y_1,Y_2$
such that $conv(Y_1) \cap conv(Y_2) \neq \emptyset$.
But $$conv(Y_1) \subseteq \bigcap_{k \in K : x_k\in Y_2} k \quad \text{and} \quad
conv(Y_2) \subseteq \bigcap_{k \in K :x_k \in Y_1} k,$$ by construction.
This is a contradiction, so $|K|<r$.
\end{proof}

We observe that 

\begin{claim}
Let $C$ be a convexity space and $B$ be its half-spaces. 
If the Radon number of $C$ is $r$ then
the VC dimension of $B$ is smaller than $r$.
\end{claim}

\begin{proof}
Let $Y \subseteq X$ be of size $r$.
The set $Y$ can thus be partitioned to $Y_1,Y_2$ so that
$conv(Y_1) \cap conv(Y_2) \neq \emptyset$.
Assume, towards a contradiction, that there is $b \in B$ so that
$b \cap Y = Y_1$.
Since $B$ consists of half-spaces\footnote{Here we use a more general definition of half-spaces,  as in the definition of locally-separable in Section~\ref{sec:ext}.},
there is $\bar b \in B \subseteq C$ so that $Y_2 = \bar b \cap Y$.
This implies that $conv(Y_1) \cap conv(Y_2) = \emptyset$,
which is a contradiction. Thus, for all $b \in B$ we have $b \cap Y \neq Y_1$
which means that the VC dimension is less than $|Y|=r$.
\end{proof}

\section{The chromatic number of the Kneser graph}
\label{sec:Alon}

Here we prove a lower bound on the chromatic number of the
Kneser graph (which is weaker than Lovasz's). 
We follow an argument of Alon~\cite{AlonPrivate}, who informed us that a similar argument
was independently found by Szemeredi.
We focus on the following case, but the argument applies
more generally.

\begin{theorem}
\label{thm:KGnn4}
For $n$ be divisible by $4$ we have
$\chi(KG_{n,n/4}) > n/10$.
\end{theorem}

The first step in the proof is the following lemma proved by Kleitman~\cite{Kleitman66}.
A family $F \subseteq 2^X$ is called intersecting if
$f \cap f' \neq \emptyset$ for all $f,f' \in F$.

\begin{lemma}
If $F_1,\ldots,F_s \subset 2^{[n]}$, where each $F_i$ is intersecting, 
then
$$\Big| \bigcup_{i \in [s]} F_i \Big| \leq 2^n-2^{n-s}.$$
\end{lemma}

The lemma can proved by induction on $s$.
The case $s=1$ just says that an intersecting family has size at most $2^{n-1}$.
The induction step is based on correlation of monotone events
(for more details see, e.g.~\cite{as}).

\begin{proof}[Proof of Theorem~\ref{thm:KGnn4}]
Consider a proper coloring of $KG_{n,n/4}$ with $s$ colors.
Let $V_1,\ldots,V_s$ be the partition of the vertices to color classes.
Each $V_i$ is an intersecting family.
Let $F_i$ be the family of sets $u \subseteq [n]$ that contain
some set in $V_i$.
Each $F_i$ is also intersecting.
By the lemma above,
$$2^n -  \Big| \bigcup_{i \in [s]} F_i \Big| \geq 2^{n-s}.$$
On the other hand, 
the complement of $\bigcup_{i \in [s]} F_i$ is of size less than
$\sum_{k=0}^{n/4}{n \choose k} \leq 2^{n H(1/4)}$,
where $H(p) = - p \log (p)-(1-p) \log(1-p)$ is the binary entropy function.
Hence,
$$s > n(1-H(1/4)) \geq n /10.$$
\end{proof}

\end{document}